\documentclass[11pt]{amsart}
\usepackage[left=2cm, right=2cm, top=2cm, bottom=2cm]{geometry}
\usepackage[utf8]{inputenc}
\usepackage{pdfpages}
\usepackage{wrapfig}
\usepackage[font=small]{caption, subcaption}

\usepackage[colorlinks=true,linkcolor=red]{hyperref}
\usepackage[square,numbers]{natbib}
\usepackage{typeset-directory/misc-cmds}
\usepackage{typeset-directory/matrices-and-vectors}

\begin{document}

\title[Duality \& Transport on Magnetic Graphs]{Kantorovich Duality and Optimal Transport Problems on Magnetic Graphs}
\author{Sawyer Jack Robertson}
\thanks{The author sincerely thanks his advisor, Javier Alejandro Ch\'{a}vez-Dom\'{i}nguez at the University of Oklahoma who has provided years of support and mentorship.}
\thanks{This research was supported in part by a National Merit Scholarship provided to the author jointly by the National Merit Scholarship Corporation and the University of Oklahoma.}
\begin{abstract}
    We consider Lipschitz- and Arens-Eells-type function spaces constructed for magnetic graphs, which are adapted to this setting from the area of optimal transport on discrete spaces. After establishing the duality between these spaces, we prove a characterization of the extreme points of the unit ball in the magnetic Lipschitz space as well as a result identifying the magnetic Arens-Eells space as a quotient of the classical Arens-Eells space of an associated classical graph called the lift graph.
\end{abstract}
\address{University of Oklahoma}
\email{sawyerjack@ou.edu}
\date{\today}

\maketitle   

\section{Introduction}
\subsection{Background}
Let $G=(V(G),E(G))$ be an undirected, finite graph without loops or multiple edges (henceforth `simple'), and suppose one has two mass (probability) distributions $\mu,\nu:V(G)\rightarrow\R$. A common question concerns how one may transport the mass distribution $\mu$ to the distribution $\nu$ in a manner which is optimal with respect to certain quantities of interest like energy or cost. Such questions constitute the research area of optimal transport on discrete domains \cite{JS, SRGB}, a topic which has applications in a number of applied areas such as computer graphics and image processing \cite{GBL, FFS, LROY}, geometry \cite{MQO}, and physics \cite{BB}. One classical approach toward these problems is by way of Kantorovich duality, which in the formulation presented here relates Lipschitz-type and Arens-Eells function spaces via duality.\par
One setting where discrete transport problems have, to this author's knowledge, not been posed is magnetic (or signed) graphs. These are essentially combinatorial graphs which have been equipped with an additional structure known as a signature, which can be viewed as a discrete analogue of a magnetic potential field \cite{LL}. These graphs have helped researchers model systems from discrete quantum mechanics \cite{LL} and chemistry \cite{FCSLP}, to even social psychology \cite{CH}. In the classical theory of optimal transport on discrete spaces, there happens to be a well-understood link between the cost of transport along paths and their associated lengths. Interestingly, it appears as though one natural extension of this relationship to the case of magnetic graphs appears to fail, which we will explore in the last section. This has complicated the computation of quantities associated to magnetic transport processes. \par
In this paper, we will approach optimal transport through adapted Lipschitz- and Arens-Eells-type function spaces designed for magnetic graphs. After some preliminary remarks, we will establish the duality of these spaces using a form of representation in the manner of \cite{NW}. Then, we will prove a characterization of the extreme points of the unit ball in our (magnetic) Lipschitz space. Finally, we will put down a result concerning the central problem of computing the $\sigma$-Arens-Eells norm via a compression mapping.\\

\subsection{Graph theory preliminaries}

\begin{wrapfigure}{R}{.275\textwidth}
    \begin{minipage}{\linewidth}
        \centering\captionsetup[subfigure]{justification=centering}
        \includegraphics[width=\linewidth]{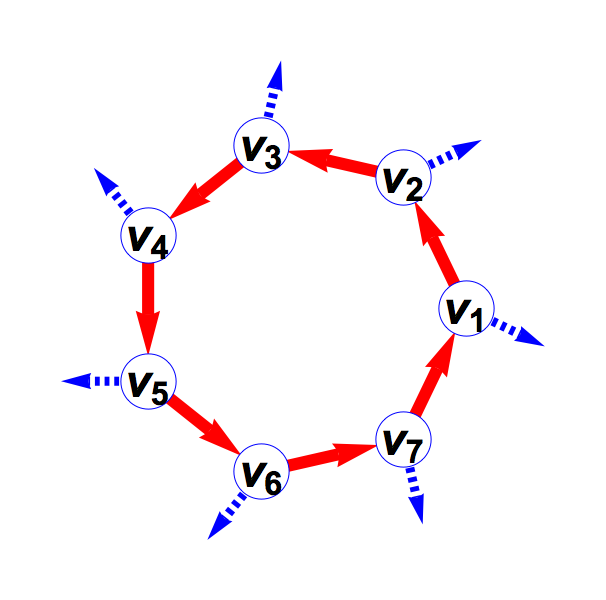}
            \subcaption{7-cycle, with signature illustrated by the angular offset of the blue arrows from the red ones}
            \label{fig:5a}\par\vfill
        \includegraphics[width=\linewidth]{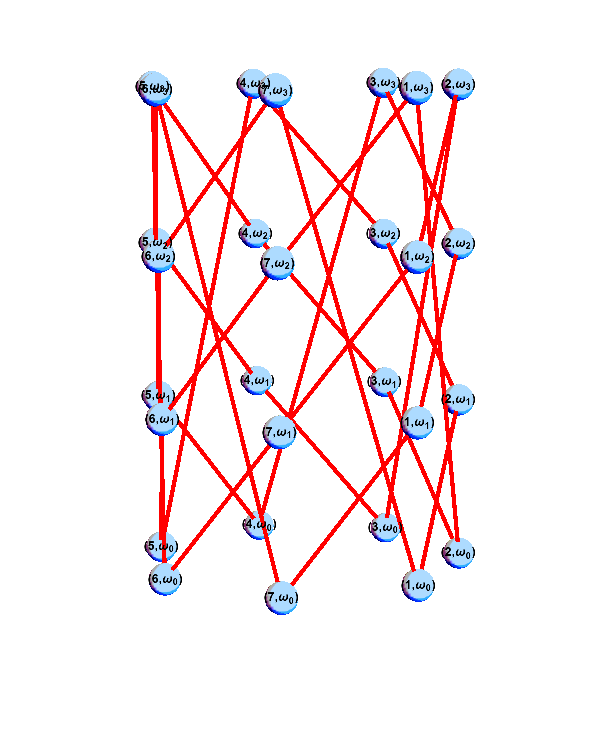}
            \subcaption{lift of the above cycle; this graph is isomorphic to a cycle on 28 vertices}
            \label{fig:5b}
    \end{minipage}
\caption{A magnetic cycle graph and its lift.}\label{fig:graph-fig}
\end{wrapfigure}

Throughout, $\mathbf{S}^1=\{z\in\C:|z|=1\}$ is the unit circle, and $\mathbf{S}^1_p=\{z\in\C:z^p=1\}$ is the abelian group of $p$-th roots of unity, where $p\in\mathbb{N}$. All graphs considered here are considered to be simple; that is, undirected, with a finite vertex set, no loops, and no multiple edges. If $u,v$ are vertices, adjacency is indicated $u\sim v$. A graph is connected if there exists a path connecting any two of the vertices in the graph.\par
If $X=(V(X),E(X))$ is a simple graph, we define the \textit{oriented edges} of $X$ to be the set $E^{\text{or}}(X):=\big\{(u,v),\hspace{0.1cm}(v,u):\{u,v\}\in E(X)\big\}$.\par
A \textit{signature} on $X$ is a map $\sigma:E^{\text{or}}(X)\rightarrow \mathbf{S}^1:(u,v)\mapsto\sigma_{uv}$, satisfying $\sigma_{vu}=\overline{\sigma_{uv}}$. A \textit{magnetic graph} is a pair $(X,\sigma)$. Throughout, (non)-magnetic graphs will be denoted with an (`$G$') `$X$' respectively. The trivial signature is defined to be 1 on every oriented edge. A magnetic graph $(X,\sigma)$ is called \textit{balanced} provided that the product of the values of the signature along any (directed) cycle is 1; otherwise, $X$ is \textit{unbalanced}. If $\sigma$ takes values in a finite subgroup of $\mathbf{S}^1$ and $\tau:V(X)\rightarrow\mathbf{S}^1_p$ is some function, then we may produce the $\tau$-\textit{switched signature} denoted $\sigma^\tau$ via
    \begin{equation}\label{switchedsig}
        \hspace{-4cm}\sigma^\tau_{uv}:=\tau(u)\sigma_{uv}\tau(v)^{-1}.
    \end{equation}
Two distinct signatures related in this manner by some switching function are called \textit{switching equivalent}. A signature $\sigma$ is balanced if and only if it is switching equivalent to the trivial signature \cite[Proposition 3.2]{LCL}. \par
Given a magnetic graph $(X,\sigma)$ whose signature takes values in some finite group $\mathbf{S}^1_p$ we may construct a related non-magnetic graph called the \textit{lift of} $X$, denoted $\widehat{X}$, via vertex set $V(\widehat{X})=V(X)\times \mathbf{S}^1_p$, and with the condition that two vertices $(u,\sigma_1),(v,\sigma_2)$ are adjacent if and only if $u\sim v$ in the original graph and $\sigma_2=\sigma_1\sigma_{uv}$. The signature structure from the original graph is thus encoded in the edge structure of the new one, illustrated in Figure \ref{fig:graph-fig}.\par
We will also have occasion to utilize the Hilbert space $\ell_2(V(X)):=\{f:V(X)\rightarrow\C\}$ with inner product structure given by
    \[\hspace{-4cm}\ip{f,g}_{\ell_2}:=\sum_{u\in V(X)}f(u)\overline{g(u)}.\]
Also, we will use the unit distributions $\delta_u\in\ell_2(X)$ given by
    \[\hspace{-4cm}\delta_u(v):=\begin{cases}
     0 & v\neq u  \\
     1 & v=u       \\
    \end{cases}.\]
    
\subsection{Classical Kantorovich duality, extreme points}
To complete this preliminary section, let us recall some results pertaining to non-magnetic graphs. If $\mu,\nu$ are two mass distributions on the vertices of a connected graph $G$ equipped with shortest-path metric $d$, we consider transport plans $\gamma:V(G)\times V(G)\rightarrow\R_{\geq 0}$ which are mass distributions on the Cartesian product of the vertex set whose marginals agree with $\mu$ and $\nu$, and such that $\gamma(u,v)$ represents the amount of mass transported from vertex $u$ to vertex $v$. $\Gamma(\mu,\nu)$ represents the set of all transport plans between $\mu$ and $\nu$. The \textit{1-Wasserstein distance} between $\mu,\nu$ is then given by
    \[W_1(\mu,\nu)=\inf_{\gamma\in\Gamma(\mu,\nu)}\sum_{u,v}\gamma(u,v)d(u,v).\]
If one chooses an arbitrary but fixed base vertex, say $u_0\in V(G)$, one can define the normed space
    \[\Lipp_0(G):=\left\{f:V(G)\rightarrow\R\setbar f(u_0)=0\right\}\]
where 
\begin{equation}\label{lip0norm}
\norm{f}_{\Lipp}=\max_{u\sim v}|f(u)-f(v)|.
\end{equation}
Similarly, for each pair of vertices $u,v\in V(G)$ we may define the \textit{combinatorial atom} $m_{uv}:V(G)\rightarrow\R$ defined by
    \[m_{uv}(w)=\begin{cases}
     1  &  w=u \\
     -1  & w=v \\
     0   &\text{otherwise}\\
    \end{cases}\text{ for each }w\in V(G).\]
Subsequently, we may construct the \textit{Arens-Eells} space via
    \[\AEc(G):=\text{span}_{\R}{\{m_{uv}\setbar u,v\in V(G),\hspace{0.25cm} u\sim v\}}\]
equipped with the norm
    \[\norm{m}_\AEc:=\inf\Big{\{}\sum_{i=1}^n|a_i|\setbar m=\sum_{i=1}^n a_im_{u_iv_i},\hspace{0.1cm}\{a_i\}_i\subset\R,\hspace{0.15cm}u_i\sim v_i\Big{\}}.\]
Viewing $\AEc(G)$ and $\Lipp_0(G)$ as subspaces of the Hilbert space $\ell_2(G)$, one can prove via Riesz representation that $\AEc(G)$ is isometrically isomorphic to $\Lipp_0(G)^\ast$, e.g. \cite[Theorem 2.2.2]{NW}. This is the so-called (classical) Kantorovich duality to which we dedicate a good part of the sequel.
\section{Duality and extreme points}

\subsection{Arens-Eells, signed Lipschitz spaces}
Let $(X,\sigma)$ be a magnetic graph, endowed with usual shortest-path metric $d$. We define the signed Lipschitz function space
    \[\lipsig(X):=\{f:V(X)\rightarrow\C\hspace{0.2cm}\big|\hspace{0.2cm}\exists \hspace{0.1cm}C\geq 0\text{ s.t. }|f(u)-\sigma_{uv}f(v)|\leq C\text{ for each }u\sim v\}\]
This definition leads to a natural choice of  $\sigma$-Lipschitz norm, which we pair with an equivalent formulation. For each $f\in\lipsig(X)$, set
    \begin{equation}\label{lipsignorm}\begin{split}
    \norm{f}_\lipsig&:=\inf\{C\geq 0\hspace{0.1cm}\big|\hspace{0.2cm}|f(u)-\sigma_{uv}f(v)|\leq C\text{ for each }u\sim v\}\\
    &=\max_{u\sim v}|f(u)-\sigma_{uv}f(v)|.
    \end{split}\end{equation}
    \begin{lemma}
        Let $(X,\sigma)$ be an unbalanced magnetic graph. Then $\norm{\cdot}_\lipsig$ is a norm.
    \end{lemma}
    \begin{proof}
        Let $f,g\in\lipsig(X)$ and $\alpha\in\C$. Clearly $\norm{\cdot}_\lipsig\geq 0$, and $\norm{\alpha f}_\lipsig=|\alpha|\norm{f}_\lipsig$ from the definition. The triangle inequality is obtained as follows:
            \[\begin{split}
            \norm{f+g}_\lipsig&=\max_{u\sim v}|(f+g)(u)-\sigma_{uv}(f+g)(v)|\leq \max_{u\sim v}|f(u)-\sigma_{uv}f(v)|+|g(u)-\sigma_{uv}g(v)| \\
            &=\max_{u\sim v}|f(u)-\sigma_{uv}f(v)|+\max_{u\sim v}|g(u)-\sigma_{uv}g(v)|=\norm{f}_\lipsig+\norm{g}_\lipsig.
            \end{split}\]
        For definiteness, let us assume that $\norm{f}_\lipsig=0$. The max formulation in equation \eqref{lipsignorm} would imply that for each pair of adjacent vertices $u,v$ one has $f(u)=\sigma_{uv}f(v)$, forcing either $f\equiv 0$ or $f=\lambda f_p$, where $|\lambda|>0$, $|f_p|\equiv 1$, and $\overline{f_p}$ is a switching function for each of the connected components of $(X,\sigma)$, such that $\sigma^{f_p}\equiv 1$ as in equation \eqref{switchedsig}. The latter case contradicts the assumption that $(X,\sigma)$ is unbalanced so $f\equiv 0$.
    \end{proof}
In the case where $(X,\sigma)$ is a balanced graph, $\norm{\cdot}_{\lipsig}$ is a semi-norm since its definiteness cannot be assured. Let us now consider two adjacent vertices $u,v\in V(X)$ and define the \textit{magnetic atom} $m^\sigma_{uv}:V(X)\rightarrow\C$ as follows:
    \[m^\sigma_{uv}(w):=
    \begin{array}{cc}
         \Bigg\{ & 
            \begin{array}{cc}
             1 & w=u \\
             -\sigma_{uv} & w=v \\
             0 & \text{otherwise} \\
            \end{array}. \\
    \end{array}\]
We define the magnetic Arens-Eells space   
    \[\AEsig(X):=\text{span}_\C\{m^\sigma_{uv}\setbar u,v\in V(X),\hspace{0.25cm} u\sim v\}.\]
Elements of this space will be called \textit{magnetic molecules}. We will use the next lemma to verify that $\AEsig(X)$ indeed recovers all of $\ell_2(X)$ under the right condition.
    \begin{lemma}
    Let $(X,\sigma)$ be an unbalanced magnetic graph. Then $\AEsig(X)=\ell_2(X)$.
    \end{lemma}
    \begin{proof}
    We will prove this lemma by showing that the orthogonal complement of $\AEsig(X)$ in $\ell_2(X)$ is merely $\{0\}$. Let $f\in\ell_2(X)$ be such that $\ip{f,\overline{m}}_{\ell_2}=0$ for each $m\in\AEsig(X)$. In particular, for each pair of adjacent vertices $u,v\in V(X)$ one has $\ip{f,\overline{m^\sigma_{uv}}}_{\ell_2}=0$. Explicitly, this means $f(u)=\sigma_{uv}f(v)$ forcing either $f\equiv 0$ or $f=\lambda f_p$, where $|\lambda|>0$ and $|f_p|\equiv 1$ is a switching function for each of the connected components of $(X,\sigma)$, such that $\sigma^{f_p}\equiv 1$ as in equation \eqref{switchedsig}. The latter case contradicts the assumption that $(X,\sigma)$ is unbalanced so $f\equiv 0$.
    \end{proof}
We define for each molecule $m\in\AEsig(X)$ the norm
    \[\norm{m}_\AEsig:=\inf\bigg\{\sum_{i=1}^n|a_i|\setbar m=\sum_{i=1}^na_im^\sigma_{u_iv_i},\hspace{0.1cm}\{a_i\}_i\subset\C ,\hspace{0.1cm} u_i\sim v_i\in V(X) \bigg\}.\]
Let us quickly check that this is actually a norm as claimed.
    \begin{lemma}\label{aesignorm}
    Let $(X,\sigma)$ be an unbalanced magnetic graph. Then $\norm{\cdot}_\AEsig$ is a norm.
    \end{lemma}
    \begin{proof}
    The positivity, homogeneity, and triangle inequality for this norm are all easily checked. We only need argue for why $\norm{\cdot}_\AEsig$ is in fact definite. Suppose for some molecule one has $\norm{m}_\AEsig=0$. For positive integer $k$, find some finite linear combination of atoms $\sum_i a_i^km^\sigma_{u_i^kv_i^k}$ for which
        \[m=\sum_i a_i^km^\sigma_{u_i^kv_i^k},\hspace{0.25cm}\sum_i|a_i^k|<\frac{1}{k}.\]
    Then,
        \[\norm{m}_{\ell_2}=\norm{\sum_i a_i^km^\sigma_{u_i^kv_i^k}}_{\ell_2}\leq \sum_i |a_i^k|\norm{m^\sigma_{u_i^kv_i^k}}_{\ell_2}\\
        \leq 2\sum_i|a_i^k|<\frac{2}{k}\rightarrow 0\text{ as }k\rightarrow\infty.\]
    From the definiteness of the $\norm{\cdot}_{\ell_2}$ norm, the claim is verified.
    \end{proof}
We have constructed two Banach function spaces for unbalanced magnetic graphs, $\lipsig(X)$ and $\AEsig(X)$ (we verified their structures as normed spaces, completeness follows from their finite dimension). In Theorem \ref{dualitysigma}, we identify them as dual to one another. \par

\subsection{Duality}
We will now adapt the classical duality result mentioned in the preliminary discussion to the function spaces designed for magnetic graphs. The argument is in the manner of Weaver \cite{NW}.

\begin{theorem}\label{dualitysigma}
    Let $(X,\sigma)$ be an unbalanced magnetic graph. Then $\AEsig(X)^\ast$ is isometrically isomorphic to $\lipsig(X)$.
    \end{theorem}
    \begin{proof}
    Let us define a linear mapping $T_1:\AEsig(X)^\ast\rightarrow\lipsig(X)$ in the following manner. Let $M\in\AEsig(X)^\ast$, and notice that since $\AEsig(X)=\ell_2(X)$ as finite dimensional vector spaces, $M$ may be viewed as a continuous linear functional on $\ell_2(X)$. In turn, by using the finite-dimensional Riesz representation theorem on the space $\ell_2(X)$, we may obtain a representative function $f_M\in\ell_2(X)$ so that for each $m\in\AEsig(X)$, one has $M(m)=\ip{f_M,\overline{m}}_{\ell_2}$. Put $T_1(M)=f_M$. Notice that for each pair of adjacent vertices $u,v\in V(X)$ we have
        \[\begin{split}
        |f_M(u)-\sigma_{uv}f_M(v)|&=|\ip{f_M,\overline{m^\sigma_{vu}}}_{\ell_2}|\\
        &=|M(m^\sigma_{vu})|\leq \norm{M}_{\AEsig^\ast}\norm{m^\sigma_{vu}}_\AEsig\leq \norm{M}_{\AEsig^\ast}.
        \end{split}\]
    In turn, by taking a max,
        \[\norm{f_M}_\lipsig=\max_{u\sim v}|f_M(u)-\sigma_{uv}f_M(v)|\leq \norm{M}_{\AEsig^*}\]
    which implies that $T_1$ is a nonexpansive operator. As a note, the linearity of $T_1$ is inherited from the Riesz Representation. Let us now suggestively define a mapping $$T_2:\lipsig(X)\rightarrow\AEsig(X)^\ast:f\mapsto M_f,$$ where for each $m\in\AEsig(X)$ we set $M_f(m)=\ip{f,\overline{m}}_{\ell_2}$. We verify that given any $m\in\AEsig(X)$, realized as a finite linear combination $\sum_i a_i m^\sigma_{u_iv_i}$, it holds that
        \[\begin{split}
            |M_f(m)|&=|\ip{f,\overline{m}}_{\ell_2}|=\big|\ip{f,\sum_i\overline{a_i m^\sigma_{u_iv_i}}}_{\ell_2}\big|\\
            &\leq\sum_i|a_i|\cdot|\ip{f,\overline{m^\sigma_{u_iv_i}}}_{\ell_2}|=\sum_i|a_i|\cdot|f(u)-\sigma_{uv}f(v)|\\
            &\leq \norm{f}_\lipsig\cdot\sum_i|a_i|.
        \end{split}\]
    By taking an inf over all possible representatons of $m$, we obtain the inequality $|M_f(m)|\leq  \norm{f}_\lipsig\cdot\norm{m}_\AEsig$, which implies $\norm{M_f}_{\AEsig(X)^\ast}\leq \norm{f}_\lipsig$, showing that $T_2$ is a non-expansive linear operator as well. The composition $T_2T_1:\AEsig(X)^\ast\rightarrow\AEsig(X)^\ast$ is easily checked to be the identity mapping. Since the mapping $T_1$ and its inverse are nonexpansive and invertible, $T_1$ is a vector space isomorphism and an isometry of Banach spaces, finalizing the claim.
    \end{proof}

\subsection{Extreme points}
Some contextual remarks are in order before presenting the result. 
    \begin{definition}
       Let $(W,\norm{\cdot}_W)$ be a normed space, and suppose $f\in W$, with $\norm{f}_W\leq 1$. Then $f$ is called an \textit{extreme point} of the unit ball in $W$, denoted $B_W$, provided that for any $g\in W$, if 
            \[\left\{f+tg\setbar t\in[-1,1]\right\}\subset B_W,\]
        then $g= 0$.
    \end{definition}
We note that this definition of extremity is equivalent to the more classical interpretation, which defines points in the unit ball to be extreme if they cannot be expressed as the midpoint of two other, distinct elements of the unit ball.
    \begin{definition}
        Let $G$ be simple graph and $f\in\Lipp_0(G)$ be an element of the unit ball of $\Lipp_0(G)$, denoted $B_\Lipp$. We say an edge $\{u,v\}\in E(X)$ is \textit{satisfied} by $f$ if
            \[|f(u)-f(v)|=1.\]
    \end{definition}
Farmer \cite[Theorem 1]{JF} proves an equivalent version of the following result.
    \begin{theorem}[Farmer 1994]
        Let $G$ be a connected simple graph and $f\in B_\Lipp$. Then $f$ is an extreme point of $B_\Lipp$ if and only if the graph $H_f$, constructed with vertex set $V(G)$ and edge set
            \[E(H_f):=\left\{\{u,v\}\in E(G)\setbar \{u,v\}\text{ is satisfied by }f\right\}\]
        is connected.
    \end{theorem}
We present an analogue of this result to the function spaces designed for magnetic graphs. First, one preliminary definition in the spirit of the preceding remarks.
    \begin{definition}
        Let $(X,\sigma)$ be an unbalanced magnetic graph, and suppose $f\in \lipsig(X)$ is in the unit ball of $\lipsig(X)$, denoted $B_\lipsig$. We say an edge $\{u,v\}\in E(X)$ is $\sigma$-\textit{satisfied} by $f$ if
            \[|f(u)-\sigma_{uv}f(v)|=1.\]
    \end{definition}
Note as before that the quantity $|f(u)-\sigma_{uv}f(v)|$ does not depend on the choice of orientation of the edge being evaluated, i.e. $|f(u)-\sigma_{uv}f(v)|=|-\sigma_{vu}f(u)+f(v)|$.
    \begin{theorem}
        Let $(X,\sigma)$ be an unbalanced graph, and $f\in B_\lipsig$. Then $f$ is an extreme point of $B_\lipsig$ if and only if the magnetic graph $H_f$ defined by the vertex set $V(X)$, the edge set
            \[E(H_f):=\left\{\{u,v\}\in E(X)\setbar \{u,v\}\text{ is $\sigma$-satisfied by }f\right\},\]
        and which we equip with the restriction of the signature structure $\sigma$ as on $X$, is unbalanced on each of its connected components.
    \end{theorem}
    \begin{proof}
        Let us begin with the converse by supposing that $H_f$ is unbalanced on each of its connected components, and that some $g\in \lipsig(X)$ satisfies
            \[\left\{f+tg\setbar t\in[-1,1]\right\}\subset B_\lipsig.\]
        This implies that for every edge $\{u,v\}\in E(X)$ $\sigma$-satisfied by $f$, it holds for every $t\in[-1,1]$
            \[|f(u)-\sigma_{uv}f(v)+t\left(g(u)-\sigma_{uv}g(v)\right)|\leq 1.\]
        Knowing that $|f(u)-\sigma_{uv}f(v)|=1$, and that 1 is an extreme point of the unit ball in $\C$, the only way that the inequality above can hold for every $t\in[-1,1]$ is if $|g(u)-\sigma_{uv}g(v)|=0$ at every $\sigma$-satisfied edge; that is, $g(u)=\sigma_{uv}g(v)$. As we have seen before, this would imply that on each connected component of $H_f$, $g$ must be either identically 0 or a scalar multiple of a switching function for $\sigma$ associated to the trivial signature. Since the latter implication contradicts our assumption that $(H_f,\sigma)$ is unbalanced on each of these connected components, it must hold that $g\equiv 0$, implying that $f$ is an extreme point for $B_\lipsig$.\par
        
        Next, let us assume that $(H_f,\sigma)$ has some balanced connected component; that is, there exists $A\subset V(X)$ such that the subgraph induced by $A$ and the existing satisfied edges between its vertices $E(A)$, and with the signature $\sigma$ restricted to its oriented edges $E^{\text{or}}(A)$, is balanced. We claim that $f$ cannot be an extreme point. Since $A$ is balanced, there exists a function $h:A\rightarrow\mathbf{S}^1$ such that for each oriented edge $(u,v)\in E^{\text{or}}(A)$, $h(u)=\sigma_{uv}h(v)$. We note that $E(H_f)$ need not contain every edge in the original graph, so let us identify 
            \[\epsilon:=\max\left\{|f(u)-\sigma_{uv}f(v)|\setbar \{u,v\}\in E(X),|f(u)-\sigma_{uv}f(v)|<1 \right\}<1.\]
        If this set happens to be empty, choose $\epsilon:=0$. Define the nonzero function $g\in\lipsig(X)$ by
            \[g(x):=\begin{cases}
            \frac{1-\epsilon}{2}h(x)&\text{ if }x\in A\\
            0&\text{ if }x\notin A\\
            \end{cases}\]
        for each $x\in V(X)$. Let us check that $\{f+tg\setbar t\in [-1,1]\}\subset B_\lipsig$. If $\{u,v\}\in E(X)$, then one of three possible cases holds: (i) both $u\in A$ and $v\in A$; (ii) both $u\notin A$ and $v\notin A$; (iii) or $u\in A$ and $v\notin A$. For either of the first two cases, it holds
            \[\begin{split}
            |f(u)-\sigma_{uv}f(v)+t\left(g(u)-\sigma_{uv}g(v)\right)|&\leq|f(u)-\sigma_{uv}f(v)|+|t|\cdot|g(u)-\sigma_{uv}g(v)|\\
            &\leq |f(u)-\sigma_{uv}f(v)|+|g(u)-\sigma_{uv}g(v)|\leq 1+0=1.
            \end{split}\] 
        Notice that in these two cases, $|g(u)-\sigma_{uv}g(v)|=0$ since either $g(u)=\sigma_{uv}g(v)$ (case (i)), or $g(u)=g(v)=0$ (case (ii)). In the third case, that is when $u\in A$ and $v\notin A$, it holds $|g(u)-\sigma_{uv}g(v)|=\frac{1-\epsilon}{2}<1$ by definition. In turn,
            \[\begin{split}
            |f(u)-\sigma_{uv}f(v)+t\left(g(u)-\sigma_{uv}g(v)\right)|&=|f(u)-\sigma_{uv}f(v)|+|t||g(u)-\sigma_{uv}g(v)|\\
            &\leq |f(u)-\sigma_{uv}f(v)|+|g(u)-\sigma_{uv}g(v)|\leq \epsilon + \frac{1-\epsilon}{2}\\
            &< \epsilon+1-\epsilon =1.
            \end{split}\]
        and the claim holds. This completes the proof.
        \end{proof}
\section{Compression}
The final result we will present concerns an approach into the computation of the norm $\norm{\cdot}_{\lipsig}$. We wish to say something about how the magnetic transport norm for molecules on the original magnetic graph may be related to the classical transport norm for molecules on the lift graph. First we need a way of translating between spaces.
    \begin{definition}
        Let $(X,\sigma)$ be a magnetic graph, and assume $\sigma$ takes values in $\mathbf{S}^1_p$ for some integer $p\geq 1$. We define the linear compression mapping $$\mathcal{C}:\AEc(\widehat{X})\rightarrow \AEsig(X)$$ as follows: for each $m\in \AEc(\widehat{X}), u\in V(X)$, put
            \[ (\mathcal{C}m)(u)=\sum_{\xi\in \mathbf{S}^1_p}\xi m(u,\xi). \]
    \end{definition}
    \begin{theorem}\label{surjcont}
        Let $(X,\sigma)$ be a magnetic graph, and assume $\sigma$ takes values in $\mathbf{S}^1_p$ for some integer $p\geq 1$. Then $\mathcal{C}$ is a surjective contraction from $\AEc(\widehat{X})$ onto $\AEsig(X)$.
    \end{theorem}

    \begin{proof}
        First, let us verify the surjectivity of the mapping $\mathcal{C}$. Suppose we take some $m^\sigma\in\AEsig(X)$, which we represent with a finite linear combination of magnetic atoms in the form
            \[m^\sigma=\sum_{i=1}^n a_i m^\sigma_{u_iv_i}.\]
        By choosing $m^\ast\in\AEc(\widehat{X})$ to be
            \[m^\ast = \sum_{i=1}^n a_i m_{(u_i,1),(v_i\sigma_{u_iv_i})},\]
        we may compute
            \[\begin{split}
            (\mathcal{C}m^\ast)(u)&=\sum_{\xi\in \mathbf{S}^1_p}m^\ast(u,\xi)=\sum_{\xi\in \mathbf{S}^1_p} \sum_{i=1}^n a_i m_{(u_i,1),(v_i\sigma_{u_iv_i})}(u,\xi)\\
            &=\sum_{i=1}^n a_i \delta_{u_i}(u)-\sigma_{u_iv_i}\delta_{v_i}(u)=\sum_{i=1}^n a_i m^\sigma_{u_iv_i}(u)=m^\sigma(u)
            \end{split}\]
        which verifies the onto claim; notice that from this computation, we obtain the general relation that for each adjacent pair of vertices $(u,\omega),(v,\omega\sigma_{uv})$ in the lift, we have 
            \begin{equation}\label{conatoms}\mathcal{C}(m_{(u,\omega),(v,\omega\sigma_{uv})})=\omega m^\sigma_{uv}.\end{equation}
        Now let $m\in\AEc(\widehat{X})$ be given, and suppose we represent it by a linear combination of atoms in the form
            \[m=\sum_{\ell=1}^m b_\ell m_{(u_\ell,\omega_\ell),(v_\ell,\omega_\ell\sigma_{u_\ell v_\ell})}.\]
        In turn, from \eqref{conatoms}, one has
            \[\mathcal{C}m=\sum_{\ell=1}^m b_\ell \omega_\ell m^\sigma_{u_\ell v_\ell}.\]
        Applying inequalities, we see
            \[\norm{\mathcal{C}m}_{\AEsig}\leq \sum_{\ell=1}^m |b_\ell\omega_\ell|= \sum_{\ell=1}^m |b_\ell|\]
        which, after taking an $\inf$ over all such representations of $m$, implies that $\mathcal{C}$ is a contraction.
    \end{proof}
As a simple corollary to the preceding theorem, we have the following equation.
    \begin{corol}\label{signormliftfibre}
        Let $m^\sigma\in \AEsig(X)$. We have the equation
            \[\norm{m^\sigma}_\AEsig=\min\big\{\norm{m}_\AEc\hspace{0.1cm}\big{|}\hspace{1mm}m\in \AEc(\widehat{X}); \mathcal{C}m=m^\sigma\big\}.\]
    \end{corol}
    \begin{proof}
        Knowing that $\mathcal{C}$ is surjective, the set on the right is nonempty; and, knowing also that $\mathcal{C}$ is a contraction, we may write
            \[\norm{m^\sigma}_\AEsig\leq\inf\big\{\norm{m}_\AEc:m\in \AEc(\widehat{X}); \mathcal{C}m=m^\sigma\big\}.\]
        By checking that the norm $\norm{m^\sigma}_\AEsig$ is attained in the set somewhere we will verify the reverse inequality and justify the use of a $\min$ in the equation. To this end, fix some $m^\sigma \in \AEsig(X)$, and realize it as an optimal linear combination of magnetic atoms, i.e.
            \[m^\sigma=\sum_{i=1}^n a_i m^\sigma_{u_iv_i},\hspace{0.3cm}\text{where }\norm{m^\sigma}=\sum_{i=1}^n |a_i|.\]
        Note that such a linear combination of molecules exists by a compactness argument. Notice that, as in the preceding proof, the molecule $m^\ast\in \AEc(\widehat{X})$ given by
            \[m^\ast = \sum_{i=1}^n a_i m_{(u_i,1),(v_i\sigma_{u_iv_i})}\]
        satisfies $\mathcal{C}m^\ast=m^\sigma$. Since the expression above is one realization of $m^\ast$ as a linear combination of (non-magnetic) atoms, it holds that $\norm{m^\ast}_\AEc\leq\sum_i |a_i|=\norm{m^\sigma}_\AEsig$. Moreover, since $\mathcal{C}m^\ast=m^\sigma$ and $\mathcal{C}$ is a contraction, it holds that $\norm{m^\sigma}_\AEsig\leq \norm{m^\ast}_\AEc$. Putting the two inequalities together, i.e.
            \[\norm{m^\ast}_\AEc\leq \sum_i |a_i|=\norm{m^\sigma}_\AEsig\leq \norm{m^\ast}_\AEc\]
        we find $\norm{m^\ast}_\AEc=\norm{m^\sigma}_\AEsig$ as desired. This completes the proof.
    \end{proof}
The author had hoped to improve this result via a conjecture concerning the link of simple magnetic molecules defined on the original graph to associated ones on the lift. We can define \textit{magnetic path molecules} to take the value 1 at some initial vertex (say `$x$'), and at some terminal vertex (say `$y$'), it takes the value of the negative of the product of the signature values along a path between the initial vertex and terminal vertex (say `$-\sigma$'). Based on the classical case, it would be natural to guess that the norm of a magnetic path molecule would coincide with the length of a path on the lift initiating at $(x,1)$ and terminating at $(y,\sigma)$. This supposes the stability of the strong relationship between the norm of path-type molecules and their lengths in the classical case as it is adapted to the magnetic case.
\par As we briefly mentioned in the introduction, such a conjecture does indeed fail. Consider the following counterexample, illustrated in Figure \ref{fig:counterexample}. The graph is on three vertices, with the signature taking the value 1 on the upper two edges and the value $\sqrt{-1}$ along the oriented edge $(u,v)$. Let us define the molecule $m\in\AEsig(X)$ via
    \[m(x)=\begin{cases}
    1  &x=u\\
    1 &x=v\\
    0  &\text{otherwise}\\
    \end{cases}.\] 
In light of the conjecture, one can view $m$ as a magnetic path molecule corresponding to the path initiating at $u$, going around the graph counterclockwise once and terminating at $v$ (this is based on the values of the function at $u$ and $v$).
    \begin{figure}[h]
        \centering
        \includegraphics[width=4cm]{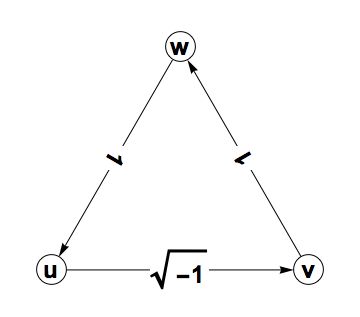}
        \caption{Counterexample graph.}
        \label{fig:counterexample}
    \end{figure}
The simplicity of the graph structure works to our advantage in the sense that any element in the space $\AEsig(X)$ has a unique representation in the basis $\{m^\sigma_{uv},m^\sigma_{vw},m^\sigma_{wu}\}$. A quick calculation yields
    \[m=(1+\sqrt{-1})m^\sigma_{uv}+m^\sigma_{vw}-m^\sigma_{wu}\]
forcing $\norm{m}_{\AEsig}=2+\sqrt{2}$, a non-integer quantity which invalidates the conjecture.

\nocite{*}
\bibliography{references}
\bibliographystyle{plainnat}

\end{document}